\documentclass[11pt, a4paper, leqno]{amsart}
\usepackage{amsfonts, amssymb, amsmath}
\usepackage{amsthm}
\usepackage{a4wide}
\usepackage{enumerate}

\newtheorem{theorem}{Theorem}
\newtheorem{proposition}{Proposition}[section]
\newtheorem{lemma}[proposition]{Lemma}
\theoremstyle{definition}
\newtheorem{remark}[proposition]{Remark}
\numberwithin{equation}{section}

\newcommand{\N}{{\mathbb N}}
\newcommand{\R}{{\mathbb R}}
\newcommand{\eps}{{\varepsilon}}

\newcommand{\A}{{\mathcal A}}
\newcommand{\I}{{\mathcal I}}
\newcommand{\J}{{\mathcal J}}
\newcommand{\F}{{\mathcal F}}
\newcommand{\D}{{\mathcal D}}

\renewcommand{\H}{{H}}

\newcommand{\Norm}[1]{\lVert #1 \rVert}
\newcommand{\abs}[1]{\lvert #1 \rvert}
\newcommand{\weakto}{\rightharpoonup}

\newcommand{\beq}{\begin{equation}}
\newcommand{\eeq}{\end{equation}}

\begin{document}

\title[Nonlinear Schr\"odinger equations with fast decaying potentials]
{Semiclassical stationary states for nonlinear Schr\"odinger equations with fast decaying potentials}
\author{Vitaly Moroz}
\address{Swansea University\\ Department of Mathematics\\ Singleton Park\\
Swansea\\ SA2~8PP\\ Wales, United Kingdom}	
\email{V.Moroz@swansea.ac.uk}
\author{Jean Van Schaftingen}
\address{Universit\'e Catholique de Louvain\\ D\'epartement de Math\'ematique\\ Chemin du Cyclotron 2\\ 1348 Louvain-la-Neuve \\ Belgium}
\email{Jean.VanSchaftingen@uclouvain.be}

\date{\today}

\begin{abstract}
We study the existence of positive solutions for a class of nonlinear Schr\"odinger equations of the type
\[-\eps^2\Delta u + Vu = u^p \quad\text{in $\R^N$}, \]
where $N\ge 3$, $p>1$ is subcritical and $V$ is a nonnegative continuous potential.
Amongst other results, we prove that if $V$ has a positive local minimum, and $\frac{N}{N-2}<p<\frac{N+2}{N-2}$,
then for small $\eps$ the problem admits positive solutions
which concentrate as $\eps\to 0$ around the local minimum point of $V$.
The novelty is that no restriction is imposed on the rate of decay of $V$.
In particular, we cover the case where $V$ is compactly supported.
\end{abstract}

\maketitle

\section{Introduction}

We study the existence of positive solutions for a class of nonlinear Schr\"odinger equations which includes,
in particular, equations of the type
\begin{equation}\label{0}
-\eps^2\Delta u + Vu = u^p \quad\text{in $\R^N$},
\end{equation}
where $p>1$, $\eps>  0$ and $V\in C(\R^N, \R^+)$ is a nonnegative potential.
Solutions of this equation are stationary states of the nonlinear Schr\"odinger equations.
The parameter $\eps$ is the adimensionalised Planck constant;
one expects to recover classical physics when $\eps$ goes to $0$.
This \emph{r\'egime} is referred to as the semiclassical limit.
Equation \eqref{0} also models the formation of spike layers in cross-diffusion \cite{WMNi}.

First results go back to Floer and Weinstein \cite{Floer}, who considered the case $N=1$ and $p=3$.
Using a Lyapunov--Schmidt reduction method, they proved that if $V$ is bounded and has a positive global nondegenerate minimum,
then for $\eps$ small enough there is a family of solutions that concentrate around the minimum point.
Oh \cite{Oh88, Oh90}, also using Lyapunov--Schmidt reduction techniques,
obtained multibump solutions, i.e.\ solutions concentrating around multiple nondegenerate critical points of $V$.
The use of variational methods was initiated by Rabinowitz \cite{R}, who proved the existence of a solution $u_\eps$ for small $\eps > 0$ under the assumption
\[ 0<\inf_{\R^N} V < \liminf_{\abs{x} \to \infty} V(x). \]
Wang Xuefeng investigated the concentration phenomenon \cite{Wa93}.
His results imply, in particular, that if solutions $u_\eps$ attain its global maximum at $x_\eps$,
then $\liminf_{\eps \to 0} u_\eps (x_\eps)>0$
and there exist $C, \lambda > 0$ such that
\[
 u_\eps(x) \le  C\exp\big(-\tfrac{\lambda}{\eps} \abs{x-x_\eps}\big).
\]
A variational method was subsequently devised by del Pino and Felmer \cite{dPF} in order to obtain solutions that concentrate around an arbitrary local minimum of $V$.

Throughout all these works, an assumption that $\inf_{\R^N} V> 0$ was made.
If $V \ge 0$ then $V$ may vanish at some points of the domain, or $V$ may vanish at infinity.
Byeon and Wang \cite{BW-I, BW-II} have studied solutions concentrating around zeros of $V$. A remarkable feature is that these solutions have specific different concentration behavior that depends on the behavior of $V$ near its zero.

The study of the case where $V>0$, but $\inf_{\R^N} V=0$ has been initiated by Ambrosetti, Felli and Malchiodi \cite{AFM}.
They have proved the existence of solutions to the related problem
\[ -\eps^2\Delta u+ Vu=Ku^p \quad\text{in $\R^N$},  \]
when $K\in C(\R^N, \R^+)$ is a nonnegative potential which decays fast enough.
Ambrosetti, Malchiodi and Ruiz \cite{AMR} have then proved, by Lyapunov--Schmidt reduction method,
the existence of solutions to \eqref{0} when $V$ satisfies the assumption
\[
  \liminf_{\abs{x} \to \infty} V(x)\abs{x}^2 > 0,
\]
which we call \emph{slow decay}.
Finally, Bonheure and Van Schaftingen \cite{BVS-CR, BVS} have proved the existence and concentration of solution to \eqref{0}
by the method of del Pino and Felmer in the case
\[
\liminf_{\abs{x} \to \infty} V(x)\abs{x}^{(N-2)(p-1)} > 0,
\]
which thus provides an improvement to the results in \cite{AMR} when $p > \frac{N}{N-2}$.

In this paper, we address the question of existence and concentration for
\emph{fast decaying potentials}, i.e.\ potentials for which
\[ \liminf_{\abs{x} \to\infty} V(x)\abs{x}^2 = 0. \]
One difference between equations with slow and fast decaying potentials is the
decay rate of solutions as $x\to\infty$. Similarly to equations with $\inf_{\R^N}V>0$,
positive solutions of \eqref{0} with slow decaying potentials such that
$\liminf_{\abs{x} \to\infty} V(x)\abs{x}^2 =\infty$ have an exponential decay at infinity.
Solutions of \eqref{0} with fast decaying $V$ may decay polynomially.
For instance, if
\begin{equation}\label{fast-delta}
\limsup_{\abs{x} \to\infty} V(x)\abs{x}^{2+\delta}=0,
\end{equation}
for some $\delta>0$, then positive solutions of \eqref{0} decay no faster then $\abs{x}^{-(N-2)}$,
as one can see by comparing with an explicit subsolution at infinity $\abs{x}^{-(N-2)}(1+\abs{x}^{-\delta})$ of $-\Delta+V$.
A consequence of such polynomial decay of solutions is a Liouville type nonexistence phenomena:
i.e., if \eqref{fast-delta} holds then equation \eqref{0}
\emph{has no positive solutions in the neighborhood
of infinity for $p\le\frac{N}{N-2}$}, cf.\ \cite{KLS}.
A special case of a fast decaying potential is a potential $V$ that vanishes identically.
In this case, the equation
$$-\Delta u = u^p\quad\text{in $\R^N$}$$
{\em has no positive solutions  for $p<\frac{N+2}{N-2}$}, see\ \cite{GS}.
The existence of positive solutions of \eqref{0} with fast decaying potentials
in the admissible range $\frac{N}{N-2}<p<\frac{N+2}{N-2}$ is thus a rather delicate issue.

A special case of our results in this paper is the following theorem,
which in particular, answers positively the question about the existence of solutions
for \eqref{0} with compactly supported potentials,
which was posed by Ambrosetti and Malchiodi \cite[Section 1.6.5]{AM}.

\begin{theorem}\label{T-0}
Let $N\ge 3$, $\frac{N}{N-2}<p<\frac{N+2}{N-2}$ and $V \in C(\R^N, \R^+)$ be a nonnegative potential.
If there exists
a smooth bounded open set $\Lambda\subset\R^N$ such that
\[
  0<\inf_{x\in\Lambda}V(x)<\inf_{x\in\partial\Lambda}V(x),
\]
then there exists $\eps_0>0$ such that for every $0<\eps<\eps_0$,
equation \eqref{0} has at least one positive solution $u_\eps$.
\end{theorem}

As a byproduct of our method, we obtain results about the concentration of solutions.
A central issue in this analysis is that, if $(u_\eps)_{\eps>0}$ is a family solutions of \eqref{0} that
concentrates to a point $x_0\in\R^N$, then
\[v_\eps(x):=u_\eps\left(x_0+\eps x\right)\]
solves the rescaled equation
\[-\Delta v_\eps+V\left(x_0+\eps x\right)v_\eps=v_\eps^p \quad\text{in $\R^N$}.\]
This suggests that $v_\eps$ should converge, in a certain sense, to a positive solution
of the limiting equation
\[-\Delta v+V\left({x_0}\right)v=v^p \quad\text{in $\R^N$}, \]
where $V(x_0)>0$.
It is known that such a solution $v$ decays exponentially as $\abs{x}\to\infty$.
On the other hand, if $V$ satisfies \eqref{fast-delta} then $v_\eps$ decays no faster then $\abs{x}^{-(N-2)}$.
Concentration results for \eqref{0} should thus capture a transition between
polynomial decay of the concentrating solutions $u_\eps$ and exponential decay of the limiting solution $v$.
Actually, we show that $u_\eps$ decays polynomially in $x$ and
exponentially in $\eps$. More precisely, we prove the following.

\begin{theorem}
Let $(u_\eps)$ be the family of solutions of \eqref{0}, constructed in Theorem~\ref{T-0}.
Then  for all sufficiently small $\eps>0$
there is $x_\eps\in\Lambda$ such that $u_\eps$ attains its maximum at $x_\eps$,
\begin{gather*}
  \liminf_{\eps \to 0} u_\eps(x_\eps) > 0, \\
  \lim_{\eps\to 0}V(x_\eps)=\inf_{x\in\Lambda}V(x),
\end{gather*}
and there exists $C, \lambda>0$ such that
\[
  u_\eps(x) \le C\exp\Big(-\frac{\lambda}{\eps}\frac{|x-x_\eps|}{1+|x-x_\eps|}\Big)\, \big(1+|x-x_\eps|^2\big)^{-\frac{N-2}{2}}.
\]
\end{theorem}

In particular, when $V$ is compactly supported, solutions $(u_\eps)$ solve the equation
\begin{equation}\label{1}
-\Delta u = u^p
\end{equation}
in an exterior domain.
Remind that, according to \cite[Theorem 3.6]{GS}, solutions to \eqref{1} in exterior domains for
$\frac{N}{N-2} < p < \frac{N+2}{N-2}$ decay at infinity either as $\abs{x}^{-2/(p-1)}$,
or as $\abs{x}^{-(N-2)}$. The family of solutions $(u_\eps)$, constructed in Theorem~\ref{T-0}, belongs to the former class.
Note also that if $V$ is compactly supported then the solutions $u_\eps$ belong to $L^2(\R^N)$ only when $N\ge 5$,
while for $N=3, 4$ we have $u_\eps \in L^{\frac{N+2}{N-2}}(\R^N)$.

Our approach in this work follows the variational penalization scheme introduced in \cite{dPF} and adapted to decaying potentials in \cite{BVS-CR, BVS}. Formally, equation \eqref{0} is the Euler--Lagrange equation of the functional
\[
  \I_\eps(u):=\frac{1}{2}\int_{\R^N}\left(\eps^2|\nabla u|^2+V|u|^2\right)-
\frac{1}{p+1}\int_{\R^N}|u|^{p+1}.
\]
The first integral defines a natural Sobolev space. However, when
$1 < p < \frac{N+2}{N-2}$, the second integral need not be finite in this space, and one has thus $\I_\eps(u) \in \R \cup \{-\infty\}$.
This difficulty can be overridden by following the method devised by del Pino and Felmer. They modified the problem for large $u$ and $x$ so that the modified problem becomes well-posed and solvable \cite{dPF}.
One has then to show that solutions of the modified problem are small enough for large $x$, so that they solve the original problem.

In particular, in order to tackle problems with decaying potentials, in \cite{BVS-CR, BVS} the penalized problem
\[
  -\eps^2 \Delta u_\eps+Vu_\eps= \chi_\Lambda u_\eps^{p-1} + \chi_{\Lambda^c} \min (\kappa V, u^{p}_\eps)u_\eps,
\]
where $0<\kappa<1$, was considered.
One has then to show that $u_\eps \le \kappa V$ outside $\Lambda$. When $V$ is compactly supported this approach fails, because one should then have that the solution are compactly supported, which cannot be the case. Our key observation in this paper is that, in order to overcome this difficulty, $V$ can be replaced in the definition of the penalization by a Hardy-type potential $H$, chosen independently of the decay of $V$.
We also improve the barriers used in \cite{BVS} to obtain the decay of $u_\eps$.

The paper is organized as follows. In Section~\ref{sectAssumptions} we give the precise assumptions and results of this paper. The three next sections are devoted to the proof of these results: the penalization problem is introduced in Section~\ref{s-Penalization}, the asymptotics of its solutions are studied in Section~\ref{s-Energy}, and in Section~\ref{s-Barriers}, the proof is completed by obtaining the decay of the solutions, and proving that solutions of the penalized problem solve the original problem. Finally, Section~\ref{s-Ext} discusses various extensions to two-dimensional problems, problems on domains and more general nonlinearities as well as improvements of some results in \cite{BVS}.

\section{Assumptions and the main result}
\label{sectAssumptions}

\subsection{Assumptions}
We consider a slightly more general equation than \eqref{0}, i.e.,
\[
\label{Peps}
  -\eps^2\Delta u + Vu = Ku^p \quad\text{in $\R^N$}, \tag{$\mathcal{P}_\eps$}
\]
where $N\ge 3$, $p>1$, $\eps>0$ and $V, K\in C(\R^N, \R^+)$ are nonnegative potentials.
The existence of solutions will be related to the presence of local minimizers of the concentration function
\[
\A(x):=\frac{V(x)^{\frac{p+1}{p-1}-\frac{N}{2}}}{K(x) ^{\frac{2}{p-1}}}.
\]
The linear part of the equation induces the norm
\[
\|u\|_\eps^2:=\int_{\R^N}\left(\eps^2|\nabla u|^2+V|u|^2\right)
\]
and the weighted Sobolev space
\[
  \D^1_V(\R^N):=\left\{u\in\D^1_0(\R^N)\, \big|\, \|u\|_\eps<+\infty\right\}.
\]
Here $\D^1_0(\R^N)$ is the closure of $C^\infty_c(\R^N)$ with respect to the $L^2$-norm of the gradient.
The space $\D^1_V(\R^N)$ endowed with the norm $\|\cdot\|_\eps$ is a Hilbert space.
Note that the set $\D^1_V(\R^N)$ does not depend on $\eps>0$.
If $V$ is compactly supported then $\|\cdot\|_\eps$
simply defines an equivalent norm on $\D^1_0(\R^N)$, while for general bounded nonnegative potentials $V$
one always has the embeddings
\[H^1(\R^N)\subseteq\D^1_V(\R^N)\subseteq\D^1_0(\R^N).\]

\subsection{Main result}
Our main result reads as follows.

\begin{theorem}\label{T-main}
Let $N\ge 3$, $1<p<\frac{N+2}{N-2}$ and let $V, K\in C(\R^N, \R^+)$. Assume that
there exists $\sigma<(N-2)p-N$ and $M >0$ such that
\[
\label{a-K}
\tag{$K$}0\le K(x)\le M(1+\abs{x})^{\sigma}\quad\text{for all $x\in\R^N$},
\]
and that there exists a smooth bounded open set $\Lambda\subset\R^N$ such that
\[
\tag{$\mathcal{A}$} 0<\inf_{x\in\Lambda}\A(x)<\inf_{x\in\partial\Lambda}\A(x).
\]
Then there exists $\eps_0>0$ such that for every $\eps\in(0, \eps_0)$,
equation \eqref{Peps} has at least one positive solution $u_\eps\in \D^1_V(\R^N)\cap C^1(\R^N)$.
Moreover,
\[\|u_\eps\|_\eps=O(\eps^{N/2})\quad\text{ as $\:\eps\to 0$}, \]
$u_\eps$ attains its maximum at $x_\eps \in \Lambda$,
\begin{gather*}
  \liminf_{\eps \to 0} u_\eps(x_\eps) > 0, \\
  \lim_{\eps\to 0}\mathcal{A}(x_\eps)=\inf_{x\in\Lambda}\mathcal{A}(x),
\end{gather*}
and there exists $C, \lambda > 0$ such that
\begin{equation}\label{fast}
u_\eps(x) \le C\exp\Big(-\frac{\lambda}{\eps}\frac{|x-x_\eps|}{1+|x-x_\eps|}\Big)\, \big(1+|x-x_\eps|^2\big)^{-\frac{N-2}{2}}.
\end{equation}
\end{theorem}

If $V$ satisfies the fast decay assumption \eqref{fast-delta}
then the restriction $\sigma<(N-2)p-N$ in the theorem is sharp,
in the sense that \eqref{Peps} has no positive solutions for $\sigma\ge (N-2)p-N$ (see e.g.\ \cite{KLS}).
If $V$ satisfies \eqref{fast-delta} then the upper bound \eqref{fast} is sharp as $\abs{x}\to\infty$ in the sense that for each fixed $\eps\in(0, \eps_0]$
\[
  \liminf_{\abs{x}\to\infty}\abs{x}^{N-2} u_\eps(x)>0.
\]
This follows, e.g., by comparison with an explicit subsolution at infinity $\abs{x}^{-(N-2)}(1+\abs{x}^{-\delta})$ of $-\Delta+V$.

The asymptotic behavior of the solutions can be described as follows. Let $(\eps_n)_{n\ge 1}$ be a sequence that decreases to zero,
and $(x_n)_{n\ge 1}\subset\Lambda$ be a sequence such that
\[\liminf_{n \to \infty} u_{\eps_n}(x_{n}) > 0\quad\text{and}\quad x_n\to\bar x\in\Lambda.\]
Then $\A(\bar x)=\inf_\Lambda\A$ and, along a subsequence, the sequence of rescaled solutions
\[v_n(x):=u_{\eps_n}(x_n+\eps x)\]
converges in the $C^1_\mathrm{loc}(\R^N)$ topology
to a positive solution $v\in H^1(\R^N)\cap C^1(\R^N)$ of the limiting equation
\[
-\Delta v+V(\bar x)v=K(\bar x) v^p \quad\text{in $\R^N$}.
\]
See Lemma~\ref{lemma-conv} below for details.

\subsection{Organization of the proof}
The proof of Theorem~\ref{T-main} is organized as follows.
In Section~\ref{s-Penalization} we introduce an adequate modification of the penalization
scheme of \cite{BVS} which allows us to include into consideration potential $V$ with fast decay.
Then we apply the mountain-pass lemma to establish the existence of a family of positive solutions $(u_\eps)$
to the penalized problem. In Section~\ref{s-Energy} we obtain energy and uniform estimates
on the mountain-pass solutions. Many of the proofs in this section require only very minor
modifications comparing to the results in \cite{BVS}, so we omit the details in most cases.
In particular, we establish in Section~\ref{s-Energy} a first weak concentration result,
Lemma~\ref{Prop-A}, which tells that solutions $u_\eps$ uniformly decay to zero as $\eps\to 0$
outside a family of balls $B(x_\eps, \eps R)$ whose centers $x_\eps$ concentrate to the local minima of
the concentration function $\A(x)$. This information becomes crucial in Section~\ref{s-Barriers},
where it is used to arrange a comparison of solutions $u_\eps$ with carefully constructed
family of barrier functions, which have sharp asymptotic both as $\eps\to 0$ and $x\to\infty$.
This allows to establish the sharp concentration bound \eqref{fast} and at the same time
to show that solutions $u_\eps$ of the modified problem actually solve
the original problem \eqref{Peps}, which completes the proof of Theorem~\ref{T-main}.

\section{Penalization scheme}\label{s-Penalization}

\subsection{Penalization potential}
Without loss of generality, we assume that $0 \in \Lambda$. One can then choose $\rho > 0$ so that $\overline{B(0, \rho)}\subset\Lambda$.
Let $\chi_\Lambda$ denote the characteristic function of the set $\Lambda$.
We define the penalization potential $\H:\R^N\to\R$ by
\begin{equation}\label{H}
\H(x):=\frac{\kappa(1-\chi_{\Lambda}(x))}{\abs{x}^2\big(\log\frac{\abs{x}}{\rho_0}\big)^{1+\beta}},
\end{equation}
where $\beta > 0$, $\rho_0>0$ and $\kappa>0$ are chosen so that $\rho_0 < \rho$ and
\[
  \frac{\kappa}{\big(\log \frac{\rho}{\rho_0}\big)^{1+\beta}} < \frac{(N-2)^2}{4}.
\]
The Hardy inequality
\[
  \int_{\R^N} \abs{\nabla u}^2 \ge \frac{(N-2)^2}{4} \int_{\R^N} \frac{\abs{u(x)}^2}{\abs{x}^2}\,dx
  \qquad\forall u\in C^\infty_c(\R^N)
\]
ensures positivity of the quadratic form associated to $-\Delta-\H$ on $\R^N$.

\begin{lemma}
\label{lemmaQuadratic}
For every $u\in\D^1_0(\R^N)$,
\begin{equation}\label{kappa-form}
\int_{\R^N}\left(|\nabla u|^2-H \abs{u}^2\right)\ge \biggl( \frac{(N-2)^2}{4}-\frac{\kappa}{\big(\log \frac{\rho}{\rho_0}\big)^{1+\beta}}\biggr) \int_{\R^N}\frac{\abs{u(x)}^2}{\abs{x}^2}\, dx.
\end{equation}
\end{lemma}

This implies, in particular, that the linear operator $-\Delta-\H$ satisfies the comparison principle
on open subdomains $G\subseteq\R^N$. We formulate it in a form
which is convenient for our purposes:

\begin{lemma}[Comparison Principle]
\label{lemmaComparison}
Let $G\subseteq\R^N$ be a smooth domain.
Assume that $u, v\in H^1_\mathrm{loc}(G)\cap C(\bar G)$ satisfy
\[
-\Delta u-H u\ge -\Delta v-H v \quad\text{in $G$},
\]
$\nabla(u-v)_-\in L^2(G)$ and $(u-v)_-\in L^2\big(G, (1+\abs{x})^{-2}dx\big)$.
If $\partial G\neq\emptyset$, assume in addition that $u\ge v$ on $\partial G$.
Then $u\ge v$ in $G$.
\end{lemma}

\begin{remark}
The integrability assumption $(u-v)_-\in L^2\big(G, (1+\abs{x})^{-2}dx\big)$ is required implicitly in the proof of the corresponding Proposition 24 in \cite{BVS}, but is not mentioned explicitly in the statement therein.
\end{remark}

Lemma~\ref{lemmaComparison} is proved by multiplying the inequation by $(u-v)_-$, integrating by parts and applying \eqref{kappa-form} (cf. \cite{Agmon, BVS, LLM}).

Now we are in a position to construct a minimal positive solution to the operator $-\Delta-H$ in the complement of $\Lambda$.

\begin{lemma}
\label{lemmaMinimalSolution}
There exists $w \in C^2 (\Lambda^c)$ such that
\begin{equation}
\label{eqMinimalSolution}
  \left\{\begin{aligned}
    -\Delta w - H w &=0 &&\text{in $\overline \Lambda^c$}, \\
    w&=1&& \text{on $\partial \Lambda$}.
  \end{aligned}
\right.
\end{equation}
and
\[
 \int_{\Lambda^c}\Big(\abs{\nabla w(x)}^2+\frac{\abs{w(x)}^2}{\abs{x}^2}\Big)dx < \infty.
\]
Moreover, there exists $0 < c <  C <\infty$ such that for every $x \in \Lambda^c$,
\[
 c\abs{x}^{-(N-2)} \le w(x) \le C \abs{x}^{-(N-2)}.
\]
\end{lemma}
\begin{proof}
First one constructs $w$ by minimizing $\int_{\Lambda^c} \abs{\nabla w}^2-H \abs{w}^2$. By classical regularity estimates, $w \in C^2(\Lambda^c)$.

Now set
\begin{equation}\label{w-eps}
W (x):=\abs{x}^{-(N-2)}\Big((N-2) \beta - \kappa\big(\log \tfrac{\abs{x}}{\rho_0}\big)^{-\beta}\Big),
\end{equation}
where $\beta>0$ is taken from \eqref{H}. Computing
\[
 -\Delta W(x)=\frac{\kappa(N-2)\beta}{\abs{x}^N(\log \frac{\abs{x}}{\rho_0})^{1+\beta} }+\frac{\kappa\beta(\beta+1)}{\abs{x}^N(\log \frac{\abs{x}}{\rho_0})^{2+\beta}}
\]
one verifies that the function $W$ is a supersolution to $-\Delta-\H$ in $\Lambda^c$. Choosing $
R$ so that $\Lambda \subset B(0, R)$, and
\[
 (N-2) \beta \left(\log \frac{R}{\rho_0}\right)^{\beta} > \kappa,
\]
one checks that $W$ is positive on $\partial B(0, R) \subset \Lambda^c$.
By the comparison principle of Lemma~\ref{lemmaComparison}, $w$ is bounded from above by a positive multiple of $W$ in $\R^N \setminus B(0, R)$. Since $W(x) \le (N-2)\beta \abs{x}^{-(N-2)}$ and $w$ is continuous on $B(0, R) \setminus \Lambda$, one obtains the desired upper bound.

On the other hand, the function
\[
  v(x):=\abs{x}^{-(N-2)}
\]
is a positive subsolution to $-\Delta-\H$ in $\Lambda^c$. Thus, by Lemma~\ref{lemmaComparison},
we obtain the bound from below.
\end{proof}

The previous propositions summarize the properties of the potential  $H$, which is chosen as a largest possible potential such that the quadratic form inequality \eqref{kappa-form} of Lemma~\ref{lemmaQuadratic}, and, as a consequence, the comparison  principle (Lemma~\ref{lemmaComparison}) hold, and the minimal positive solution of \eqref{eqMinimalSolution} decays at infinity as $\abs{x}^{-(N-2)}$ (Lemma~\ref{lemmaMinimalSolution}).
Notice however that the asymptotics of the minimal positive solution only plays a role in Section~\ref{s-Barriers}, for the construction of barrier functions.

\subsection{Penalized nonlinearity.}
Define the truncated nonlinearity $g_\eps:\R^N\times\R^+\to\R$ by
\begin{equation}\label{nonlin}
g_\eps(x, s):=\chi_\Lambda(x)K(x)s^p+\min\bigl(\eps^2 \H(x), \, K(x)s^{p-1}\bigr)s.
\end{equation}
Define also $G_\eps(x, s):=\int_0^s g_\eps(x, t)\, dt$. The function $g_\eps$ is a Carath\'eodory function that satisfies the following properties:
\smallskip
\begin{enumerate}
\item[$(g_1)$] $g_\eps(x, s)=o(s)$ as $s\to 0^+$ uniformly on compact subsets of $\R^N$;

\item[$(g_2)$] $g_\eps(x, s)=O(s^p)$ as $s\to\infty$ uniformly on compact subsets of $\R^N$;

\item[$(g_3)$] $0\le (p+1)G_\eps(x, s)\le sg_\eps(x, s)$ for $(x, s)\in\Lambda\times\R^+$;

\item[$(g_4)$]$0\le 2G_\eps(x, s)\le sg_\eps(x, s)\le \eps^2 H(x)s^2$ for $(x, s)\in \Lambda^c\times\R^+$.

\end{enumerate}
We are now in a position to introduce the penalized functional
\[
\J_\eps(u):=\frac{1}{2}\int_{\R^N}\left(\eps^2|\nabla u(x)|^2+V(x)|u(x)|^2\right)\, dx-
\frac{1}{p+1}\int_{\R^N}G_\eps(x, u(x))\, dx.
\]
Using $(g_2)$, $(g_4)$ and Hardy's inequality, it is standard to check that $\J_\eps$ is well-defined
and that $\J_\eps\in C^1(\D^1_V(\R^N), \R)$.
Moreover, critical points of $\J_\eps$ are weak solutions of the equation
\begin{equation}
\label{Pteps}
-\eps^2\Delta u  + V(x) u = g_\eps(x, u) \quad\text{in $\R^N$}.\tag{$\Tilde{\mathcal{P}}_\eps$}
\end{equation}
One can also see that $0$ is a strict local minimum of $\J_\eps$
and that $\J_\eps$ is unbounded from below (cf. \cite[Lemma 5]{BVS});
so, $\J_\eps$ has the Mountain Pass geometry.
We are going to show that $\J_\eps$ satisfies the Palais--Smale condition.

\begin{lemma}
Let $(u_n)\subset\D^1_V(\R^N)$ be a Palais-Smale--sequence for $\J_\eps$, i.e., for some $c \in \R$,
\[
\J_\eps(u_n)\to c, \text{ and } \qquad \J_\eps^\prime(u_n)\to 0.
\]
If $1 < p < \frac{N+2}{N-2}$, then, up to a subsequence, $(u_n)$ converges strongly to $u \in \D^1_V(\R^N)$.
\end{lemma}
\begin{proof}
It is standard to verify using $(g_3)$ and $(g_4)$ that $(u_n)$ is bounded in $\D^1_V(\R^N)$.
Up to a subsequence, $u_n \weakto u \in\D^1_V(\R^N)$. By Rellich's theorem, one has thus $u_n \to u$ in $L^{p+1}_{\mathrm{loc}}(\R^n)$.

Further, one has, by Hardy's inequality, for $R>\rho_0$,
\[
 \int_{\R^N \setminus B(0, R)} H \abs{u}^2 \le \frac{\kappa}{\big(\log \frac{R}{\rho_0}\big)^{1+\beta}} \int_{\R^N} \frac{\abs{u(x)}^2}{\abs{x}^2}\, dx
 \le \frac{\kappa }{\big(\frac{N-2}{2}\big)^2 \big(\log \frac{R}{\rho_0}\big)^{1+\beta}} \int_{\R^N} \abs{\nabla u}^2.
\]
Since $(u_n)$ is bounded, for every $\delta > 0$ there exists $R > 1$ such that for $n \in \N$,
\begin{equation}
\label{e-2}
 \int_{\R^N \setminus B(0, R)} H\abs{u_n}^2 \le \delta.
\end{equation}
One has now, by assumption,
\[
\begin{split}
\limsup_{n \to \infty} \Norm{u_n-u}^2_\eps
&=\limsup_{n \to \infty} \int_{\R^N} (g_\eps(x, u(x))-g_\eps(x, u_n(x)))(u(x)-u_n(x))\,dx\\
&\le \limsup_{n \to \infty} \int_{B(0, R)} (g_\eps(x, u(x))-g_\eps(x, u_n(x)))(u(x)-u_n(x))\,dx \\
&\qquad+\limsup_{n \to \infty} \int_{\R^N\setminus B(0, R)} \eps^2 H (\abs{u}^2+\abs{u_n}^2) \\
& \le 2\eps^2 \delta.
\end{split}
\]
Since $\eps > 0$ is fixed and $\delta > 0$ is arbitrary, this proves the claim.
\end{proof}

\begin{remark}
The same arguments prove that the mapping $u \mapsto g(\cdot, u(\cdot))$ is completely continuous from $\mathcal{D}^1_V(\R^N) \to \mathcal{D}^1_V(\R^N)^*$, i.e.\ maps  weakly convergent sequences to strongly convergent sequences.
This simplifies the proof of the existence comparing with previous penalizations, where the corresponding mapping was not completely continuous, which made the Palais--Smale condition more delicate to establish \cite{BVS, dPF}.
\end{remark}

Since $\J_\eps$ satisfies the Palais--Smale condition, all the assumptions of the Mountain Pass Lemma are fulfilled. We obtain the following existence result for modified problem \eqref{Pteps}.

\begin{proposition}\label{MPL}
Let $1<p<\frac{N+2}{N-2}$. Set
\[\Gamma_\eps:=\{\gamma\in C([0, 1], \D^1_V(\R^N))\mid \gamma(0)=0, \J_\eps(\gamma(1))<0\}.\]
For every $\eps>0$, the minimax level
\[
  c_\eps:=\inf_{\gamma\in\Gamma_\eps}\max_{t\in[0, 1]}\J_\eps(\gamma(t))>0,
\]
is a critical value of $\J_\eps$.
\end{proposition}

We call every critical point $u\in\D^1_V(\R^N)$
such that $\J_\eps(u)=c_\eps$ a \emph{least energy} solution of \eqref{Pteps}.

By the standard regularity theory, if $u\in H^1_\mathrm{loc}(\R^N)$ is a solution of \eqref{Pteps}, then $u\in W^{2, q}_\mathrm{loc}(\R^N)$ for every $q\in(1, \infty)$.
In particular, $u\in C^{1, \alpha}_\mathrm{loc}(\R^N)$ for every $0<\alpha<1$.
In general, no further regularity can be expected
as $g_\eps$ is not continuous.
Also, by the strong maximum principle, any nontrivial nonnegative solution $u\in C^ {1,\alpha}_\mathrm{loc}(\R^N)$ of \eqref{Pteps} is strictly positive in $\R^N$.

\section{Asymptotics of solutions}\label{s-Energy}

\subsection{Upper estimate on the energy}

For every $x_*\in\Lambda$, define the functional $\F_{x_*}:H^1(\R^N)\to\R$ by
\[
  \F_{x_*}(u):=\frac{1}{2}\int_{\R^N}\left(|\nabla u|^2+V(x_*)|u|^2\right)-\frac{1}{p+1}\int_{\R^N}K(x_*) |u|^{p+1}.
\]
Set
\[
  \Gamma_0:=\{\gamma\in C([0, 1], H^1(\R^N))\mid \gamma(0)=0 \text{ and } \F_0(\gamma(1))<0\}
\]
and consider a minimax level
\[
  c_{x_*}:=\inf_{\gamma\in\Gamma_0}\max_{t\in[0, 1]}\F_0(\gamma(t)).
\]
By a scaling argument,
\[
c_{x_*}=\frac{(S_{p+1})^r}{r} \A (x_*),
\]
where
\[
  \frac{1}{r}=\frac{1}{2}-\frac{1}{p+1}
\]
and
\[
  S_{p+1}^2:=\inf\left\{\int \abs{\nabla u}^2+\abs{u}^2\mid \int_{\R^N} \abs{u}^{p+1}=1, \, u\in C^\infty_c(\R^N)\right\}
\]
is the Sobolev embedding constant. Moreover $c_{x_*}$ is a critical value of $\F_{x_*}$ (see, e.g., \cite{Willem}).

Critical points $v\in H^1(\R^N)$ such that $\F_{x_*}(v)=c_{x_*}$ are called \emph{ground states} of the equation
\begin{equation}\label{rad}
 -\Delta v + V(x_*) v =K(x_*) v^p \quad\text{in $\R^N$}.
\end{equation}
These ground states decay exponentially at infinity, i.e.,
\[
  v(x)\le C(1+\abs{x}^2)^{\frac{1-N}{4}}\exp(-\sqrt{V(x_*)}\abs{x}),
\]
with $C>0$ (see \cite[Proposition 4.1]{GNN}).
It is also known that, up to a translation, every positive ground state of \eqref{rad} is radial and radially decreasing, and that radial positive ground state is unique \cite{Kwong}.

A starting point in our consideration is a comparison between critical levels $c_\eps$ and $c_{x_*}$
for $\eps$ small and $x_*$ a local minimizer of $\A$.

\begin{lemma}\label{lemmaUpper}
If $(u_\eps)_{\eps>0}$ is a family of least energy solutions of \eqref{Pteps}, then
\[
 \limsup_{\eps \to 0}\eps^{-N}\mathcal{J}_\eps(u_\eps) \le \frac{(S_{p+1})^r}{r}  \inf_{\Lambda} \mathcal{A}.
\]
Moreover, there exists $C>0$ such that
\[
\|u_\eps\|_\eps\le C\eps^{N/2}.
\]
\end{lemma}

\begin{proof}
The proof of the first part is identical to the proof of Lemma 12 in \cite{BVS},
because all calculations are performed inside $\Lambda$ and do not depend on the choice of penalization.
The proof of the second statements follows from the first one and from $(g_3)$.
\end{proof}

\subsection{No uniform convergence to zero.}
An important pointwise information about the least energy solutions of \eqref{Pteps} is that $u_\eps$ can not converge uniformly to zero as $\eps\to 0$.

\begin{lemma}\label{l-delta}
If $u_\eps\in\D^1_V(\R^N)$ is a weak positive solution of \eqref{Pteps}, then
\[
  \Norm{u_\eps}_{L^\infty(\Lambda)} > \inf_{x \in \Lambda} \Bigl(\frac{V(x)}{K(x)}\Bigr)^\frac{1}{p-1}.
\]
\end{lemma}

\begin{proof}
Let $\delta$ denote the right-hand side of the inequality. By continuity and positivity of $V$ and $K$, $\delta > 0$.
Assume now by contradiction that $u_\eps$ is a positive solution of \eqref{Pteps} and that $u_\eps \le \delta$ on $\Lambda$.
Then for $x \in \Lambda$ one has
\[
 g_\eps(x, u_\eps(x)) \le \delta^{p-1} K(x) u_\eps(x) \le V(x) u_\eps(x).
\]
Therefore one has
\[
\begin{aligned}
 -\eps^2 \Delta u_\eps + V u_\eps &\le \chi_\Lambda V u_\eps + \eps^2 H u_\eps &&\text{on $\R^N$},
\end{aligned}
\]
and hence
\[
\begin{aligned}
 -\Delta u_\eps - H u_\eps &\le 0&&\text{on $\R^N$},
\end{aligned}
\]
Now, since $u_\eps \in \mathcal{D}^1_0(\R^N)$, Lemma~\ref{lemmaComparison}  is applicable.
One concludes that $u_\eps =0$, which brings a contradiction since $\J_\eps(u_\eps)=c_\eps > 0$ by Proposition~\ref{MPL}.
\end{proof}

\subsection{Lower estimate on the energy.}
Following \cite{BVS}, we can examine the behavior of least energy solutions $u_\eps$ along a sequence of points at which it does not vanish.

\begin{lemma}
\label{lemmaConcentration}
Let $(u_\eps)_{\eps > 0}$ be least energy solutions of \eqref{Pteps}.
Let $(\eps_n)_{n\ge 1}$ be a sequence that decreases to zero.
Let $K\ge 1$ and, for $i \in \{1, \dotsc, K\}$,
let $(x_n^i)_{n \ge 1}$ be a sequence in $\Lambda$.
If for every $i \in \{1, \dotsc, K\}$,
\[
  \liminf_{n \to \infty} u_{\eps_n}(x^i_n) > 0,
\]
and for every $i, j \in \{1, \dotsc, K\}$ such that $i \ne j$,
\[
 \lim_{n \to \infty} \frac{\abs{x^i_n-x^j_n}}{\eps_n} =+\infty,
\]
then
\[
  \liminf_{n \to \infty}  \eps_n^{-N}\mathcal{J}_{\eps_n}(u_{\eps_n}) \ge \lim_{n \to \infty} \sum_{i=1}^K \frac{S_{p+1}^r}{r} \mathcal{A}(x_n^i).
\]
\end{lemma}
\begin{proof}
This lemma is proved similarly to Proposition~16 in \cite{BVS}. The only difference is that the penalization is not the same, and that $V$ may vanish. However, the modified penalization is stronger, and $V$ does not vanish on a neighborhood of $\Lambda$. Therefore, the proof of \cite{BVS} applies straightforwardly, provided the intermediate lemmas are restated by adding conditions that sequences of points are taken in $\Lambda$.
\end{proof}

As a consequence of Lemma~\ref{lemmaConcentration}, we prove that least energy solutions $u_\eps$
concentrate around a family of points inside $\Lambda$. This is a first crude concentration result
which will be the starting point to finer concentration estimates.

\begin{lemma}\label{Prop-A}
Let $(u_\eps)_{\eps > 0}$ be least energy solutions of \eqref{Pteps}.
Let $(x_\eps)_{\eps > 0}$ be such that $x_\eps \in \Lambda$ and
\[
 \liminf_{\eps \to 0} u_\eps(x_\eps) > 0.
\]
Then
\[\liminf_{\eps \to 0} d(x_{\eps}, \partial \Lambda)> 0.\]
and
\begin{equation}\label{delta-bound}
  \lim_{\substack{\eps \to 0\\ R \to \infty}} \Norm{u_\eps}_{L^\infty(\Lambda \setminus B(x_\eps, \eps R))} =0.
\end{equation}
\end{lemma}
\begin{proof}
For the first assertion, assume by contradiction that there exists a sequence $(\eps_n)_{n \ge 1}$ such that $\eps_n\to 0$, and $\lim_{n \to \infty} d(x_{\eps_n}, \partial\Lambda)=0$.
Then, by Lemma~\ref{lemmaConcentration},
\[
S_p \inf_{\Lambda} \mathcal{A}\ge \liminf_{n \to \infty} \eps_n^{-N} \mathcal{J}_{\eps_n}(u_{\eps_n})\ge S_p \liminf_{n \to \infty} \mathcal{A}(x_{\eps_n}) \ge S_p \inf_{\partial \Lambda} \mathcal{A} > S_p \inf_{\Lambda} \mathcal{A}.
\]
But this contradicts Lemma~\ref{lemmaUpper}.

For the second assertion, assume by contradiction that there exist sequences $(\eps_n)_{n \ge 1}$ and $(y_n)_{n \ge 1}$ such that $y_n \in \Lambda$,
\begin{align*}
  \lim_{n \to \infty} \eps_n&=0, &
  u_{\eps_n}(y_n) &\ge \delta, & &\text{and}&
  \lim_{n \to \infty} \frac{\abs{x_{\eps_n}-y_n}}{\eps} &=+\infty.
\end{align*}
Then, by Lemma~\ref{lemmaConcentration},
\[
  \liminf_{n \to \infty} \eps_n^{-N} \mathcal{J}_{\eps_n}(u_{\eps_n})\ge S_p \liminf_{n \to \infty} \big(\mathcal{A}(x_{\eps_n})+\mathcal{A}(y_n)\big) \ge 2 S_p  \inf_{\Lambda} \mathcal{A}.
\]
Since $\inf_{\Lambda} \mathcal{A}> 0$, one obtains again a contradiction by Lemma~\ref{lemmaUpper}.
\end{proof}

\subsection{Convergence of rescaled solutions.}
A consequence of Lemma~\ref{lemmaConcentration} and the upper bound of Lemma~\ref{lemmaUpper} is that a sequence of least energy solutions,
rescaled along a sequence of points at which it does not vanish, converges to a solution of the limit equation.
Following the arguments in the proof of Lemma~13 and Proposition~18 of \cite{BVS},
one can establish the following.

\begin{lemma}\label{lemma-conv}
Let $(u_\eps)_{\eps > 0}$ be least energy solutions of \eqref{Pteps}.
Let $(\eps_n)_{n\ge 1}$ be a sequence that decreases to zero,
and $(x_n)_{n\ge 1}\subset\Lambda$ be a sequence such that
\[\liminf_{n \to \infty} u_{\eps_n}(x_{n}) > 0\quad\text{and}\quad x_n\to\bar x\in\Lambda.\]
Then $\A(\bar x)=\inf_\Lambda\A$ and the sequence of rescaled solutions
\[
  v_n(x):=u_{\eps_n}(x_n+\eps x)
\]
converges in $C^1_\mathrm{loc}(\R^N)$ to a positive solution $v\in H^1(\R^N)\cap C^1(\R^N)$ of the limiting equation
\[
  -\Delta v+V(\bar x)v=K(\bar x) v^p \quad\text{in $\R^N$}.
\]
\end{lemma}

In particular, using \eqref{delta-bound},
we conclude from Lemma~\ref{lemma-conv} that
\[
  \sup_{\eps>0}\|u_\eps\|_{L_\infty(\R^N)}<\infty.
\]

\section{Barrier functions and solutions of the original problem}\label{s-Barriers}

In this section we introduce barrier functions which will be used to obtain
sharp decay estimates on the least energy solutions $(u_\eps)$, and hence to show that
$(u_\eps)$ indeed solves the original problem \eqref{Peps}.

\subsection{Linear inequations outside small balls.}
Let $u_\eps\in\D^1_V(\R^N)$ be a nonnegative solution of \eqref{Pteps}.
Then, according to the construction of the penalized nonlinearity, $u_\eps$
is a subsolution of the original problem \eqref{Peps}, i.e.,
\[
-\eps^2\Delta u_\eps + V u_\eps \le K u_\eps^p \quad\text{in $\R^N$}.
\]
At the same time $u_\eps$ satisfies the linear inequation
\begin{equation}\label{L-sub}
-\eps^2\Delta u_\eps - \eps^2 H u_\eps +V u_\eps\le 0 \quad\text{in $\Lambda^c$}.
\end{equation}
The next lemma shows that a slightly weaker inequation holds outside small balls
centered around a sequence of points at which $u_\eps$ does not vanish.

\begin{lemma}\label{Prop-36}
Let $(u_\eps)_{\eps > 0}$ be least energy solutions of \eqref{Pteps}.
Let $(x_\eps)_{\eps > 0}$ be such that $x_\eps \in \Lambda$ and
\[
 \liminf_{\eps \to 0} u_\eps(x_\eps) > 0.
\]
For every $\nu\in(0, 1)$, there exists $\eps_0>0$ and $R>0$ such that for all $\eps\in(0, \eps_0)$,
\begin{equation}\label{e-36}
-\eps^2\Delta u_\eps-\eps^2 H u_\eps + (1-\nu)V u_\eps\le 0 \quad\text{in $\R^N\setminus B(x_\eps, \eps R)$}.
\end{equation}
\end{lemma}

\begin{proof}
Set
\begin{equation}\label{delta-0}
 \delta_0:=\inf_{x \in \Lambda} \Bigl(\nu\frac{V(x)}{K(x)}\Bigr)^\frac{1}{p-1}.
\end{equation}
Since $V$ and $K$ are continuous and $V$ does not vanish on $\Bar{\Lambda}$, $\delta_0 > 0$.
By Lemma~\ref{Prop-A}, there exists $\eps_0>0$ and $R>0$ such that for all $\eps\in(0, \eps_0]$ one has
\begin{equation}\label{A}
u_\eps(x)\le \delta_0\quad\text{for all}\quad x\in \Lambda\setminus B(x_\eps, \eps R).
\end{equation}
Hence,
\[
  -\eps^2\Delta u_\eps + (1-\nu)V(x)u_\eps\le  -\eps^2\Delta u_\eps + \bigl( V-Ku_\eps^{p-1}\bigr)u_\eps=0
\quad\text{in $\Lambda\setminus B(x_\eps, \eps R)$}.
\]
Further, \eqref{L-sub} implies that $u_\eps$ satisfies the desired inequality in $\Lambda^c$,
which completes the proof.
\end{proof}

\subsection{Barrier functions}
Lemma~\ref{Prop-36} suggests that one can obtain upper bounds for the family of least energy solutions $(u_\eps)$
by comparing them with appropriate supersolutions.
Following \cite{BVS}, we introduce suitable barrier functions.

\begin{lemma}\label{l-barrier}
Let $(x_\eps)_\eps\subset\Lambda$ be such that $\liminf_{\eps\to 0} d(x_\eps, \partial \Lambda)>0$, let $\nu\in(0, 1)$ and let $R > 0$. Then, there exists $\eps_0 > 0$ and a family of functions $(W_\eps)_{0 < \eps < \eps_0}$ in $C^{1, 1}(\R^N \setminus B(x_\eps, \eps R))$ such that, for $\eps \in (0, \eps_0)$,
\begin{enumerate}[i)]
 \item $W_{\eps}$ satisfies the inequation
\[
-\eps^2\Delta W_{\eps}-\eps^2 H W_{\eps}+(1-\nu)V W_{\eps}\ge 0
\quad\text{in $\R^N\setminus B(x_\eps, \eps R)$},
\]
 \item $\nabla W_\eps\in L^2(\R^N\setminus B(x_\eps, \eps R))$,
 \item $W_\eps= 1$ on $\partial B(x_\eps, \eps R)$.
 \item for every $x \in \R^N \setminus B(x_\eps, \eps R)$,
\[
  W_\eps(x) \le C\exp\Big(-\frac{\lambda}{\eps}\frac{|x-x_\eps|}{1+|x-x_\eps|}\Big)\, \big(1+\abs{x}^2\big)^{-\frac{N-2}{2}}.
\]
\end{enumerate}
\end{lemma}

In the language of \cite{BVS}, the first three properties mean that $(W_\eps)_{\eps > 0}$ is a \emph{family of barriers functions}.

\begin{proof}
Fix $\mu > 0$ so that
\[
 \mu^2 < (1-\nu) \inf_\Lambda V
\]
and choose $r$ such that
\[
 0 < r <\frac{1}{2}\liminf_{\eps \to 0}d(x_\eps, \partial \Lambda)
\]
Define for $y \in \R^N$,
\begin{equation}
\label{eqgamma}
  \gamma_\eps(y)= \cosh \frac{\mu(r-\abs{y})}{\eps}.
\end{equation}
One has on $B(0, r)$,
\[
 -\eps^2 \Delta \gamma_\eps +\mu^2 \gamma_\eps \ge 0.
\]
Let $w \in C^2(\Lambda^c)$ be the minimal positive solution to $-\Delta-H$ in $\Lambda^c$, given by Lemma~\ref{lemmaMinimalSolution}.
Let $\Tilde{w} \in C^2(\R^N)$ be a positive extension such that $\Tilde{w}(x)=1$ if
$d(x, \Lambda^c)\ge r$.
Set now
\begin{equation}
\label{eqw}
  w_\eps (x)=
\begin{cases}
 \gamma_\eps(x-x_\eps) & \text{if $x \in B(x_\eps, r)$}, \smallskip\\
 \Tilde{w}(x) & \text{if $x \in B(x_\eps, r)^c$}.
\end{cases}
\end{equation}
Now, if $\eps$ is small, $B(x_\eps, 2r) \subset \Lambda$, so that $w_\eps\in C^{1, 1}(\R^N)$.
Moreover, in $B(x_\eps, r)\setminus\{x_\eps\}$ we have
\[
 -\eps^2 \Delta w_\eps-\eps^2 H w_\eps+(1-\nu)V w_\eps \ge -\eps^2 \Delta \gamma_\eps+(1-\nu)(\inf_\Lambda V) \gamma_\eps \ge 0.
\]
One also has in $\Lambda \setminus \overline{B(x_\eps, r)}$
\[
 -\eps^2 \Delta w_\eps-\eps^2 H w_\eps+(1-\nu)V w_\eps =-\eps^2 \Delta \Tilde{w}+(1-\nu)V\Tilde{w} \ge 0,
\]
for $\eps>0$ small enough, since $\Tilde{w}>0$ and $V>0$ on $\overline{\Lambda}$.
On the other hand,  one has in $\Bar{\Lambda}^c$
\[
 -\eps^2\Delta w_\eps -\eps^2 H w_\eps+(1-\nu) V w_\eps =\eps^2(-\Delta w-Hw)+(1-\nu) V w \ge 0,
\]
since $w>0$ and solves \eqref{eqMinimalSolution}. Finally, since $w_\eps\in C^{1, 1}(\R^N)$, we conclude that
\[
  -\eps^2 \Delta w_\eps-\eps^2 H w_\eps +(1-\nu) V w_\eps \ge 0
\]
weakly in $\R^N\setminus\{x_\eps\}$.
Setting
\[
  W_\eps(x)=\frac{w_\eps(x)}{\cosh \mu(\frac{r}{\eps}-R)},
\]
one can check the other properties.
\end{proof}

As a consequence of the previous lemma, we obtain an upper bound on the family of solutions $(u_\eps)$.
\begin{proposition}\label{P-barrier}
Let $(u_\eps)_{\eps > 0}$ be least energy solutions of \eqref{Pteps}.
Let $(x_\eps)_{\eps > 0}$ be such that
\[
 \liminf_{\eps \to 0} u_\eps(x_\eps) > 0.
\]
Then there exists $C, \lambda>0$ and $\eps_0>0$ such that for all $\eps\in(0, \eps_0)$,
\begin{equation}\label{e-barrier}
u_\eps(x) \le C\exp\Big(-\frac{\lambda}{\eps}\frac{|x-x_\eps|}{1+|x-x_\eps|}\Big)\, \big(1+\abs{x}^2\big)^\frac{-(N-2)}{2},
\qquad x\in\R^N.
\end{equation}
\end{proposition}

The bound \eqref{e-barrier} implies the weaker bound
\begin{equation}\label{e-barrier-short}
u_\eps(x)\le Ce^{-\frac{\lambda}{\eps}}|x-x_\eps|^{-(N-2)},
\qquad x\in\Lambda^c,
\end{equation}
which can be sometimes more convenient to use.

\begin{proof}
By Lemma~\ref{Prop-36}, there exists $\eps_0>0$ and $R>0$ such that for all $\eps\in(0, \eps_0)$
the solutions $u_\eps$ satisfy inequation \eqref{e-36} and, by \eqref{A}, one has
\[u_\eps(x)\le \delta_0\quad\text{on }\:\partial B(x_\eps, \eps R), \]
where $\delta_0>0$ is defined by \eqref{delta-0}.
Now let $(W_\eps)_\eps$ be the family of barrier functions constructed in Lemma~\ref{l-barrier}.
By Lemma~\ref{lemmaComparison}, we conclude that
\[
  u_\eps(x) \le \delta_0 W_\eps(x)\quad\text{in $\R^N\setminus B(x_\eps, \eps R)$}.
\]
Estimating $W_\eps$ from above and taking into account that $\sup \|u_\eps\|_{L^\infty(\R^N)}<\infty$,
we obtain \eqref{e-barrier}.
\end{proof}

\subsection{Solutions of the original problem and proof of Theorem~\ref{T-main}.}

The proof of Theorem~\ref{T-main} now follows from Proposition~\ref{P-barrier} and the following.

\begin{proposition}
Let $(u_\eps)_{\eps > 0}$ be least energy solutions of \eqref{Pteps}.
If
\[\sigma<(N-2)p-N.\]
then there exists $\eps_0>0$ such that, for every $0<\eps<\eps_0$,
$u_\eps$ solves the original problem \eqref{Peps}.
\end{proposition}

\begin{proof}
By Lemma~\ref{l-delta}, there is a family of points $(x_\eps)_{\eps > 0}\subset\Lambda$ such that
\[
 \liminf_{\eps \to 0} u_\eps(x_\eps) > 0.
\]
Let $d_0:=\inf d(x_\eps, \partial \Lambda) > 0$.
Therefore, by assumption~\eqref{a-K}, Proposition~\ref{P-barrier} and \eqref{e-barrier-short}, for small $\eps>0$ and
for $x\in\Lambda^c$ we obtain
\[
 \begin{split}
K(x)\big(u_\eps(x)\big)^{p-1}&\le M(1+\abs{x})^{\sigma}\big(Ce^{-\frac{\lambda}{\eps}}\abs{x}^{-(N-2)}\big)^{p-1}\\
&\le C M e^{-\frac{\lambda }{\eps}(p-1)} (1+\abs{x})^{-(N-2)(p-1)+\sigma}\\
&\le \frac{\eps^2\kappa}{\abs{x}^{2}\big(\log \frac{\abs{x}}{\rho}\big)^{1+\beta}}=\eps^2 H(x).
\end{split}
\]
By construction of the penalized nonlinearity $g_\eps$, one has then $g_\eps\big(x, u_\eps(x)\big)=K(x)\big(u_\eps(x)\big)^p$, and therefore $u_\eps$ solves the original problem \eqref{Peps}.
\end{proof}

\begin{remark}
The proof of the preceding proposition shows that assumption \eqref{a-K} can be replaced by the existence of $M > 0$ and $\beta > 0$ such that
\[
\label{a-K+}
\tag{$K^\prime$}
  K(x) \le M\frac{(1+\abs{x})^{(N-2)p - N}}{\big(\log (\abs{x}+3)\big)^{1+\beta}}\quad\text{for all $x\in\R^N$}.
\]
\end{remark}

\section{Variants and Extensions}\label{s-Ext}

\subsection{Dimension two.}
With minor adjustments the penalization techniques developed in the paper could be modified for the case $N=2$.
Recall that the classical Hardy inequality fails on $\R^2$, i.e., if for every $u\in C^\infty_c(\R^2)$
\[
  \int_{\R^2}|\nabla u|^2\ge \int_{\R^2}H(x)\abs{u}^2,
\]
and $H\ge 0$ on $\R^2$, then $H=0$ on $\R^2$.
As a consequence, the space $\D^1_0(\R^2)$ is not well-defined, see \cite{Pinchover} for a discussion.

The following inequality can be seen as a replacement of Hardy inequality for the exterior domains on the plane:
if $\rho_0>0$, then
\begin{equation}\label{log-2}
  \int_{B(0, \rho_0)^c}|\nabla u|^2\ge \frac{1}{4}\int_{B(0, \rho_0)^c}\frac{\abs{u(x)}^2}{\abs{x}^2\big(\log\frac{\abs{x}}{\rho_0}\big)^2}\,dx
  \qquad\forall u\in C^\infty_c(\overline B(0, \rho_0)^c).
\end{equation}
To define the energy space, and to formulate the variational problem on the whole of $\R^2$,
we will need another Hardy type inequality, which is valid on the whole of $\R^2$.

\begin{lemma}
Let $\rho_0>0$ and $\rho>\rho_0$.
Then there exists $C>0$ such that for every $u\in C^\infty_c(\R^2)$,
\begin{equation}\label{Hardy-2}
\int_{\R^2}|\nabla u|^2+C\int_{B(0, \rho)\setminus B(0, \rho_0)}\abs{u}^2\ge \frac{1}{4}\int_{B(0, \rho)^c}\frac{\abs{u(x)}^2}{\abs{x}^2\big(\log\frac{\abs{x}}{\rho_0}\big)^2}\,dx.
\end{equation}
\end{lemma}

\begin{proof}
Let $\theta\in C^2(\R)$ be such that $\theta(t)\ge 1$ for $t\in\R$, $\theta(t)=1$ if $t\le 0$ and $\theta(t)=\sqrt{t}$ if $t\ge 1$.
Define $W\in C^2(\R^2)$ by
$$
W(x)=\theta\bigg(\frac{\log \frac{\abs{x}}{\rho_0}}{\log \frac{\rho}{\rho_0}}\bigg).
$$
By the Agmon--Allegretto--Piepenbrick positivity principle
(see \cite{Agmon} or \cite[Lemma A.9]{LLM}), we have for each $u\in C^\infty_c(\R^2)$,
\begin{equation}\label{AAP}
\int_{\R^2}|\nabla u|^2=\int_{\R^2}\frac{-\Delta W}{W}\abs{u}^2 + \int_{\R^2}\Big|\nabla \Big(\frac{u}{W}\Big)\Big|^2 W^2
\ge\int_{\R^2}\frac{-\Delta W}{W}\abs{u}^2.
\end{equation}

A direct computation shows that
\[
  -\Delta W(x)=-\theta''\bigg(\frac{\log \frac{\abs{x}}{\rho_0} }{\log \frac{\rho}{\rho_0}}\bigg)\frac{1}{\abs{x}^2\big (\log \frac{\rho}{\rho_0} \big)^2}.
\]
In particular, when $\abs{x}>\rho$ we have
\[
  \frac{-\Delta W(x)}{W(x)}=\frac{1}{4\abs{x}^2\big(\log \frac{\abs{x}}{\rho_0}\big)^2}.
\]
When $\rho_0<\abs{x}<\rho$ we have
\[\frac{-\Delta W(x)}{W(x)}\le \frac{C'}{\rho_0^2\big(\log \frac{\rho}{\rho_0}\big)^2}.\]
Finally, when $\abs{x}<\rho_0$ we have
\[\frac{-\Delta W(x)}{W(x)}=0.\]
Hence \eqref{Hardy-2} follows from \eqref{AAP} with $C=C'\rho_0^{-2}(\log \frac{\rho}{\rho_0})^{-2}$.
\end{proof}

Assume that $\overline{B(0, \rho)}\subset\Lambda$ and fix $\rho_0\in(0, \rho)$.
If $\eps>0$ satisfies
\begin{equation}\label{eps-2}
\eps^2 C\le\inf_{B(0, \rho)}V(x),
\end{equation}
where $C$ is the constant in \eqref{Hardy-2}, then for every $u\in C^\infty_c(\R^2)$,
\begin{equation}\label{Hardy-2V}
  \int_{\R^2} \eps^2 \abs{\nabla u}^2 + V \abs{u}^2 \ge \frac{1}{4}\int_{\Lambda^c} \frac{\abs{u(x)}^2}{\abs{x}^2\big(\log \frac{\abs{x}}{\rho_0}\big)^2}\, dx.
\end{equation}
The energy space $\D^1_V(\R^N)$ can be constructed similarly to the case $N>2$
as the closure of $C^\infty_c(\R^2)$ with respect to the norm $\|\cdot\|_\eps$ defined by
\[
\|u\|_\eps^2:=\int_{\R^2}\left(\eps^2|\nabla u|^2+V\abs{u}^2\right).
\]
For $\eps < \eta$, $\Norm{u}_{\eps} \le \Norm{u}_\eta \le \frac{\eta}{\eps} \Norm{u}_\eps$; hence all the norms define the same space for every $\eps > 0$, regardless of whether $\eps$ satisfies \eqref{eps-2}.

The penalization potential $\H:\R^N\to\R$ is defined by
\begin{equation}\label{H-2}
\H(x):=\frac{\kappa(1-\chi_{\Lambda}(x))}{\abs{x}^2\big(\log\frac{\abs{x}}{\rho_0}\big)^{2+\beta}},
\end{equation}
where $\beta > 0$ and $\kappa>0$ are chosen so that
\[
  \frac{\kappa}{\big(\log \frac{\rho}{\rho_0}\big)^{\beta}} < \frac{1}{4}.
\]
Note the exponent $2+\beta$ in \eqref{H-2} which replaces $1+\beta$ in \eqref{H}.

Inequality \eqref{Hardy-2V} ensures positivity of the quadratic form associated to $-\eps^2\Delta-\eps^2\H+V$ on $\R^N$,
for $\eps>0$ satisfying \eqref{eps-2}:
\begin{equation}\label{kappa-form-2}
\int_{\R^N}\left(\eps^2|\nabla u|^2-\eps^2H \abs{u}^2+V\abs{u}^2\right)\ge
\eps^2\Bigl( \frac{1}{4}-\frac{\kappa}{(\log \frac{\rho}{\rho_0})^{\beta}}\Bigr) \int_{B(0, \rho)^c}\frac{\abs{u}^2}{\abs{x}^2\bigl(\log\frac{\abs{x}}{\rho_0}\bigr)^2}\, dx,
\end{equation}
for every $u\in\D^1_V(\R^2)$.
This implies, in particular, that the linear operator $-\eps^2\Delta-\eps^2\H+V$ satisfies the comparison principle
on open subdomains $G\subseteq\R^N$.

\begin{lemma}[Comparison Principle]
\label{lemmaComparison-2}
Let $G\subseteq\R^2$ be a smooth domain.
Assume that $u, v\in H^1_\mathrm{loc}(G)\cap C(\bar G)$ satisfy
\[
-\eps^2\Delta u-\eps^2H u+V u \ge -\eps^2\Delta v-\eps^2H v+ V v\quad\text{in $G$},
\]
$\nabla(u-v)_-\in L^2(G)$ and $(u-v)_-\in L^2\big(G\setminus B(0, \rho), \abs{x}^{-2}\big(\log(\frac{\abs{x}}{\rho_0})\big)^{-2} dx\big)$,
where $\rho>\rho_0$.
If $\partial G\neq\emptyset$, assume in addition that $u\ge v$ on $\partial G$.
Then $u\ge v$ in $G$.
\end{lemma}

Now we construct a minimal solution to $-\Delta-H$ in the complement of $\Lambda$.

\begin{lemma}
\label{lemmaMinimalSolution-2}
There exists $w \in C^2 (\Lambda^c)$ such that
\begin{equation}
\label{eqMinimalSolution-2}
  \left\{\begin{aligned}
    -\Delta w - H w &=0 &&\text{in $\overline\Lambda^c$}, \\
    w&=1&& \text{on $\partial\Lambda$}.
  \end{aligned}
\right.
\end{equation}
and
\[
 \int_{\Lambda^c} \abs{\nabla w(x)}^2+\frac{\abs{w(x)}^2}{\abs{x}^{2}\big(\log \frac{\abs{x}}{\rho_0}\big)^{2}}\, dx  < \infty.
\]
Moreover, there exists $0 < c <  C <\infty$ such that for every $x \in \Lambda^c$,
\[
 c \le w(x) \le C.
\]
\end{lemma}
\begin{proof}
The existence of a solution follows from \eqref{log-2} by the classical variational techniques.
To obtain the asymptotic, set
\begin{equation}\label{w-eps-2}
W (x):=\beta(\beta+1) - \frac{\kappa}{\bigl(\log \tfrac{\abs{x}}{\rho_0}\bigr)^{\beta}},
\end{equation}
where $\beta>0$ is taken from \eqref{H-2}. Since, for $x \in \R^2 \setminus B(0,\rho_0)$,
\[
 -\Delta W(x)=\frac{\kappa\beta(\beta+1)}{\abs{x}^2\big(\log \frac{\abs{x}}{\rho_0}\big)^{2+\beta}}
\]
$W$ is a supersolution to $-\Delta-\H$ in $\Lambda^c$. Moreover, $W$ is positive on $\partial B(0, R) \subset \Lambda$ with $\Lambda \subset B(0, R)$, provided $R$ satisfies
\[
 \beta(\beta+1) \big(\log \tfrac{R}{\rho_0}\big)^{\beta} > \kappa.
\]
The proof continues as in Lemma \ref{lemmaMinimalSolution},
using comparison principle of Lemma~\ref{lemmaComparison-2} (or
alternatively, a comparison principle that follows from \eqref{log-2} instead of \eqref{H-2}).
\end{proof}

After these adjustments are introduced, one defines penalized nonlinearity using \eqref{nonlin}
and proceeds as in the proof of Theorem~\ref{T-main}, with obvious modifications.
The barrier functions are defined using \eqref{eqgamma}, \eqref{eqw} with $N=2$.
In this way one obtains the following result.

\begin{theorem}\label{T-main-2}
Let $p>1$ and let $V, K\in C(\R^2, \R^+)$. Assume that
there exists $\sigma<-2$ and $M >0$ such that
\[
0\le K(x)\le M(1+\abs{x})^{\sigma}\quad\text{for all $x\in\R^2$}.
\]
If there exists a smooth bounded open set $\Lambda\subset\R^2$ such that
\[
0<\inf_{x\in\Lambda}\A(x)<\inf_{x\in\partial\Lambda}\A(x).
\]
then there exists  $\eps_0>0$ such that for every $0<\eps<\eps_0$,
equation \eqref{Peps} has at least one positive solution $u_\eps\in \D^1_V(\R^2)\cap C^1(\R^2)$.
Moreover,
\[\|u_\eps\|_\eps=O(\eps)\quad\text{ as $\:\eps\to 0$}, \]
$u_\eps$ attains its maximum at $x_\eps \in \Lambda$,
\begin{gather*}
  \liminf_{\eps \to 0} u_\eps(x_\eps) > 0, \\
  \lim_{\eps\to 0}\mathcal{A}(x_\eps)=\inf_{x\in\Lambda}\mathcal{A}(x),
\end{gather*}
and there exist $C, \lambda>0$ such that
\begin{equation}
  u_\eps(x) \le C\exp\Big(-\frac{\lambda}{\eps}\frac{|x-x_\eps|}{1+|x-x_\eps|}\Big), \qquad x\in\R^2.\label{fast-2}
  \end{equation}
\end{theorem}

If $V$ has compact support in $\R^2$
then the restriction on the admissible range of $\sigma$ is sharp:
\eqref{Peps} has no positive solutions for $\sigma\ge -2$ (see e.g.\ \cite{LLM}).
The upper bound \eqref{fast-2} is optimal as $\abs{x}\to\infty$ in the sense that
\[
\label{uepsinfty}
\liminf_{\abs{x}\to\infty}u_\eps(x)>0,
\]
for each fixed $\eps\in(0, \eps_0]$. Indeed, by Lemma~\ref{lemmaComparison-2},
solutions $(u_\eps)$ can be bounded from below by a constant.
On the other hand, for every fixed $x\in\Lambda^c$, $u_\eps(x)$ tends exponentially to $0$ as $\eps \to 0$.
Note also that solutions $(u_\eps)$ do not belong to $L^p(\R^N)$ for any $1\le p<\infty$.

\subsection{Equations on domains}
With some adjustments the techniques developed in the paper can be extended to the equations
\begin{equation}\label{P-Omega}
-\eps^2\Delta u + V u = K u^p\quad\text{in $\Omega$},
\end{equation}
where $\Omega$ is a domain in $\R^N$.
Let $d_{\partial \Omega}(x):=d(x, \partial\Omega)$ denotes the distance to the boundary of $\Omega$.
For the sake of simplicity we limit our discussion to the case when $\Omega$ is bounded and smooth.

\begin{theorem}\label{t-main-domain}
Let $\Omega\subset\R^N$ be a domain with a smooth bounded boundary, $1 < p < \frac{N+2}{N-2}$ if $N\ge 3$
or $1<p<\infty$ if $N=1, 2$.
Assume that $V, K \in C(\Omega,\R^+)$ and there exists $\sigma < p + 1$ and $M > 0$ such that
\[
  K(x)\le \frac{M}{\big(d_{\partial \Omega}(x)\big)^\sigma}.
\]
Assume there exists a smooth bounded open subset $\Lambda\subset\Omega$ such that
\[
  0<\inf_{x\in\Lambda}\A(x)<\inf_{x\in\partial\Lambda}\A(x).
\]
Then there exists $\eps_0>0$ such that for every $0<\eps<\eps_0$,
equation \eqref{Peps} has at least one positive solution $u_\eps\in H^1_0(\Omega)$.
Moreover,
$u_\eps$ attains its maximum at $x_\eps \in \Lambda$,
\begin{gather*}
  \liminf_{\eps \to 0} u_\eps(x_\eps) > 0, \\
  \lim_{\eps\to 0}\mathcal{A}(x_\eps)=\inf_{x\in\Lambda}\mathcal{A}(x),
\end{gather*}
and there exists $C, \lambda >0$ such that
\[
   u_\eps(x) \le C \exp \Big(-\frac{\lambda}{\eps} \frac{\abs{x-x_\eps}}{1+\abs{x-x_\eps}}\Big)d_{\partial \Omega}(x), \qquad x\in\Omega.
\]
\end{theorem}

Recall that if $\Omega$ is a domain with a smooth bounded boundary then the classical
Hardy inequality reads as follows. There exists a constant
$C_\Omega\in(0, 1/4]$ such that
\begin{equation}\label{e-Hardy}
\int_\Omega|\nabla u|^2\, dx\ge C_\Omega\int_\Omega\frac{\abs{u}^2}{d_{\partial \Omega}{}^2}\, dx
\qquad\forall\:u\in C^\infty_c(\Omega).
\end{equation}
If $\Omega$ is convex then $C_\Omega=1/4$. In general, $C_\Omega$ varies with the domain
and could be arbitrary small.
To prove Theorem \ref{t-main-domain}, one defines the penalization potential
\[
  H (x) = \frac{\kappa(1-\chi_\Lambda(x))}{(d_{\partial \Omega}(x))^2\eta(x)},
\]
where $\eta \in C(\Omega)$ is continuous, with $0 \le \eta \le 1$ in $\Omega$,
and
\[
  \eta(x)=\big\vert\log d_{\partial \Omega}(x)\bigr\vert^{1+\beta},
\]
on a neighborhood of $\partial \Omega$,
$0<\beta<1/2$ and where $0 < \kappa < C_\Omega$.
Then for every $u\in W^{1,2}_0(\Omega)$,
\begin{equation}\label{H-d}
\int_{\Omega} \abs{\nabla u}^2\ge\kappa\int_{\Omega} H \abs{u}^2,
\end{equation}
After that, one proceeds similarly to the proof of Theorem~\ref{T-main}.

In the construction of barriers, one replaces Lemma~\ref{lemmaMinimalSolution} by
\begin{lemma}
If $w \in C(\Bar{\Omega}\setminus  \Lambda) \cap W^{1, 2}_0(\Omega \setminus  \Bar{\Lambda})$ solves
\[
 \left\{
 \begin{aligned}
   -\Delta w-H w &=  0 && \text{ in $\Omega \setminus\overline\Lambda $}, \\
   w &= 1 && \text{ on $\partial \Lambda$}, \\
   w &= 0 && \text{ on $\partial\Omega$}.
 \end{aligned}
 \right.
\]
then there exist $0<c<C<\infty$ such that for $x \in \Bar{\Omega} \setminus \Lambda$
\[
c\, d(x, \partial \Omega)\le  w(x)\le C\, d(x, \partial \Omega).
\]
\end{lemma}

\begin{proof}
First set $U_\delta=\{ x \in \Omega : d_{\partial \Omega}(x) < \delta\}$. There exists $\delta' > 0$ such that $d_{\partial \Omega}  \in C^2(\Bar{U}_{\delta'})$. One has then $\abs{\nabla d_{\partial \Omega}}=1$, and
\[
 K=\Norm{\Delta d_{\partial \Omega}}_{L^\infty(\Bar{U}_\delta)} < \infty.
\]
Define now, for $t \ge 0$,
\[
  J(t)=t\Big(\beta-\frac{1}{\bigl(\log \frac{1}{t}\bigr)^\beta}\Big).
\]
A direct computation shows that, there exists $\delta > 0$ such that, $\delta < \delta'$, and for $t \in (0, \delta)$,
\[
 -J''(t)-\frac{J(t)}{t^2\Bigl(\log \frac{1}{t}\Bigr)^{\beta+1}} \ge K \abs{J'(t)}.
\]
One can thus define $W \in C^2(\Bar{U}_{\delta'})$ by
\[
 W(x)=J\big(d_{\partial \Omega}(x)\big).
\]
One has
\[
  -\Delta W -HW=-J'(d_{\partial \Omega}) \Delta d_{\partial \Omega}-J''(d_{\partial \Omega}) \abs{\nabla d_{\partial \Omega}}^2-H J(d_{\partial \Omega}) \ge 0.
\]
Now, one can take $\mu \in \R$ such that $\mu W \ge w$ on $\Omega \cap \partial U_\delta$. By the comparison principle which follows from
\eqref{H-d}, one has $w(x) \le \mu W(x) \le \mu \beta d(x, \partial \Omega)$.

For the lower bound, the same comparison principle implies that $w\ge 0$.
Hence $-\Delta w\ge 0$ in $\Omega\setminus\overline\Lambda$ and the conclusion follows.
\end{proof}

\subsection{Slowly decaying potentials revisited.}
Although our results above are valid for all nonnegative potentials $V$,
they become sharp only in the case of compactly supported or fast decaying potentials, see discussion after Theorems~\ref{T-main} and Theorem~\ref{T-main-2}.
However, the modifications to the penalization method of \cite{BVS} that are made in this paper allow us
to improve the concentration results of \cite{BVS} for slowly decaying potentials as well.
\begin{theorem}\label{t-main-slow}
Let $N\ge 2$ and $1 < p < \frac{N+2}{N-2}$.
Assume that there exists $\alpha <2$, $\sigma>0$, $m > 0$ and $M > 0$ such that $V, K \in C(\R^N,\R^+)$ satisfy, for all $x\in\R^N$,
\begin{gather*}
 V(x)\ge m(1+\abs{x})^{-\alpha}, \\
 K(x)\le M\exp\big(\sigma \abs{x}^{1-\alpha/2}\big).
\end{gather*}
Assume that there exists a smooth bounded open set $\Lambda\subset\R^N$ such that
\[0<\inf_{x\in\Lambda}\A(x)<\inf_{x\in\partial\Lambda}\A(x).\]
Then there exists $\eps_0>0$ such that for every $0<\eps<\eps_0$,
equation \eqref{Peps} has at least one positive solution $u_\eps\in \D^1_0(\R^N)$.
Moreover, $u_\eps$ attains its maximum at $x_\eps\in\Lambda$,
\begin{gather*}
  \liminf_{\eps \to 0} u_\eps(x_\eps) > 0, \\
  \lim_{\eps\to 0}\mathcal{A}(x_\eps)=\inf_{x\in\Lambda}\mathcal{A}(x),
\end{gather*}
and there exist $C, \lambda>0$ such that
\[
  u_\eps(x) \le C \exp \Big(-\frac{\lambda}{\eps} \frac{\abs{x-x_\eps}}{(1+\abs{x-x_\eps})^{\alpha/2}}\Big), \qquad x\in\R^N.
\]
\end{theorem}

The essential improvement in this result comparing to \cite{BVS} is that now, for every fixed $x \in \R^N$, one has the optimal exponential decay rate as $\eps \to 0$.
A similar decay estimate was already obtained in \cite[Lemma 22]{AFM}.
Also note that Theorem \ref{t-main-slow} includes concentration when $\alpha < 0$; this case was not specifically addressed in \cite{AFM, BVS}.

To prove Theorem \ref{t-main-slow}, one proceeds as in the proof of Theorem~\ref{T-main},
modifying appropriately the construction of barrier functions by taking, in the proof of Lemma~\ref{l-barrier},
\[
  w_\eps(x)=\begin{cases}
              \gamma_\eps(x-x_\eps) & \text{if $\abs{x-x_\eps} \le r$}, \\
              \exp\big(\frac{\lambda}{\eps} (r^{1-\frac{\alpha}{2}}-\abs{x-x_\eps}^{1-\frac{\alpha}{2}})\big) & \text{if $\abs{x-x_\eps} > r$}.
            \end{cases}
\]

In the borderline case $\alpha=2$, one obtains the following.

\begin{theorem}
Let $N\ge 2$, $1 < p < \frac{N+2}{N-2}$.
Assume that there exists $\sigma>0$, $m > 0$ and $M > 0$ such that $V, K \in C(\R^N,\R^+)$ satisfy, for all $x\in\R^N$,
\begin{gather*}
 V(x)\ge m(1+\abs{x})^{-2}\\
 K(x)\le M(\abs{x}+1)^{\sigma}.
\end{gather*}
Assume that there exists a smooth bounded open set $\Lambda\subset\R^N$ such that
\[0<\inf_{x\in\Lambda}\A(x)<\inf_{x\in\partial\Lambda}\A(x).\]
Then  $\eps_0>0$ such that for every $0<\eps<\eps_0$,
equation \eqref{Peps} has at least one positive solution $u_\eps\in\D^1_0(\R^N)$.
Moreover, $u_\eps$ attains its maximum at $x_\eps\in\Lambda$,
\begin{gather*}
  \liminf_{\eps \to 0} u_\eps(x_\eps) > 0, \\
  \lim_{\eps\to 0}\mathcal{A}(x_\eps)=\inf_{x\in\Lambda}\mathcal{A}(x),
\end{gather*}
and there exists $C, \lambda, \nu >0 $ such that
\[
 u_\eps(x) \le C \exp \Big(-\frac{\lambda}{\eps} \frac{\abs{x-x_\eps}}{1+\abs{x-x_\eps})}\Big)\big(1+\abs{x-x_\eps}\big)^{-\frac{\nu}{\eps}}, \qquad x\in\R^N.
\]
\end{theorem}

Again, for every fixed $x \in \R^N$, one has the optimal exponential decay rate as $\eps \to 0$. In the proof, one now takes
\[
  w_\eps(x)=\begin{cases}
              \gamma_\eps(x-x_\eps) & \text{if $\abs{x-x_\eps} \le r$}, \\
              \Bigl(\dfrac{r}{\abs{x-x_\eps}}\Bigr)^{\frac{\nu}{\eps}} & \text{if $\abs{x-x_\eps} > r$}.
            \end{cases}
\]

\subsection{More general nonlinearities.}
All statements and proofs given in this paper could be extended to the equation
\[
 -\eps^2 \Delta u +Vu=Kf(u),
\]
where $f\in C(\R)$ satisfies the assumptions of \cite{BVS}, i.e.,
\begin{enumerate}
 \item[$(f_1)$] there exists $q \in (1, \frac{N+2}{N-2})$ such that $f(s)=O(s^q)$ as  $s\to 0^+$;
 \item[$(f_2)$] there exists $p\in(1, \frac{N+2}{N-2})$ such that $f(s)=O(s^p)$ as $s\to\infty$;
 \item[$(f_3)$] there exists $\theta > 2$ such that
 \[0 < \theta F(s) \le s f(s), \quad\text{for every $s>0$}, \]
 where $F(u):=\int_0^u f(s)\, ds$;
 \item[$(f_4)$] the function $s \mapsto f(s)/s$ is nondecreasing for all $s \ge 0$.
\end{enumerate}
The penalization potential $\H:\R^N\to\R$ could be then chosen as before,
the truncated nonlinearity $g_\eps:\R^N\times\R^+\to\R$ should be defined as
$$
g_\eps(x, s):=\chi_\Lambda(x)K(x)f(s)+\min\big(\eps^2 H(x)s, \, K(x)f(s)\big),
$$
while the function $\A(x)$ is to be replaced by the concentration function ${\mathcal C}(x)$, as introduced in \cite{BVS}.
The condition on $\sigma$ becomes $\sigma < (N-2)q -N$.
We leave the details to the interested reader.

\section*{Acknowledgements}

JVS was supported by the SPECT programme of ESF (European Science Foundation), the Fonds de la Recherche scientifique--FNRS, the Fonds sp\'eciaux e Recherche (Universit\'e Catholique de Louvain) and by the British Council Partnership Programme in Science (British Council/CGRI-DRI/FNRS). The authors thank Jonathan Di Cosmo for discussions.

\end{document}